\documentclass[12pt,reqno]{article}
\usepackage{graphicx,indentfirst}
\usepackage{amsmath,amssymb,mathrsfs}
\usepackage{amsthm,amscd}
\usepackage{verbatim}
\usepackage{appendix}
\usepackage{enumitem,titletoc}
\usepackage[utf8]{inputenc}
\usepackage{imakeidx}
\makeindex[columns=2, title=Alphabetical Index]
\usepackage{fancyhdr}
\usepackage{amsfonts,color}
\usepackage[all]{xy}
\usepackage{tikz-cd}
\usepackage{syntonly}
\usepackage{fancyhdr}
\usepackage{array}
\usepackage[left=2.5cm,right=2.5cm,bottom=3cm,top=3cm]{geometry}

\usepackage{tikz}
\usepackage{extarrows}
\usepackage{hyperref}

\def\XXint#1#2#3{{\setbox0=\hbox{$#1{#2#3}{\int}$ }
\vcenter{\hbox{$#2#3$ }}\kern-.6\wd0}}

\newtheorem{claim}{Claim}
\newtheorem{lem}{Lemma}
\newtheorem{prop}{Proposition}[section]

\newtheorem{corr}{Corollary}[section]

\newtheorem{remark}{Remark}
\newtheorem{theorem}{Theorem}[section]

\allowdisplaybreaks

\newcommand{\eps}{\varepsilon}

\newcommand{\ssubset}{\subset\joinrel\subset}

\makeindex

\hypersetup{
    colorlinks=true,
    citecolor=green,
    filecolor=green,
    linkcolor=blue,
    urlcolor=black
}

\numberwithin{equation}{section}

\title{Generalized Monge-Amp\`ere functionals and related variational problems}
\author{\footnote{F.T. is supported by the Center for Mathematical Sciences and Applications at Harvard University.} Freid Tong and Shing-Tung Yau}
\date{\today}
\newcommand{\Addresses}{{
		\bigskip
		\footnotesize
		\textsc{Center for Mathematical Sciences and Applications, Harvard University, 20 Garden St, Cambridge, MA 02138}\par\nopagebreak
		\textit{E-mail address}: \texttt{ftong@cmsa.fas.harvard.edu}\\
		\textsc{Department of Mathematics, Tsinghua University, Haidian District, Beijing 100084, China}\par\nopagebreak
		\textit{E-mail address}: \texttt{styau@tsinghua.edu.cn}
}}
\begin{document}

\maketitle
\begin{abstract}
In this paper, we introduce a family of real Monge-Amp\`ere functionals and study their variational properties. We prove a Sobolev type inequality for these functionals and use this to study the existence and uniqueness of some associated Dirichlet problems. In particular, we prove the existence of solutions for a nonlinear eigenvalue problem associated to this family of functionals. 
\end{abstract}

\section{Introduction}
The Monge-Amp\`ere equation is one of the most important nonlinear PDEs in geometry. It arises very naturally in geometric problems in affine geometry \cite{Calabi, Cheng-Yau4, Cheng-Yau3}, the classical Minkowski problem \cite{Pogorelov, Nirenberg, Cheng-Yau2}, and is also related to K\"ahler geometry via the complex Monge-Amp\`ere equation \cite{Yau, Cheng-Yau4}. The following Dirichlet problem for Monge-Amp\`ere equation is a subject that has been widely studied. 
\begin{equation}\label{eq: MA-equation}
\begin{cases}
\det D^2u = F(x, u, \nabla u) &\text{ in }\Omega\\
u = 0 &\text{ on }\partial\Omega. 
\end{cases}
\end{equation}
 Here we assume the function $F(x, u, \nabla u)$ is positive, and we would like to obtain convex solutions $u$. The equation \eqref{eq: MA-equation} was first studied by Pogorelov \cite{Pogorelov2}. Cheng and the second author solved \eqref{eq: MA-equation} when $F$ is independent of $\nabla u$ and the sign of $F_u$ is positive \cite{Cheng-Yau}, in that case the solutions are unique by the maximum principle. In that work, $F$ is allowed to become singular near the boundary, and the solutions are obtained in $C^{\infty}(\Omega)\cap C(\overline{\Omega})$. Under further regularity conditions for $F$, Caffarelli-Nirenberg-Spruck \cite{CNS}, and Krylov \cite{Krylov} independently obtained boundary estimates of all order for \eqref{eq: MA-equation}, and hence was able to obtain the smoothness of solutions up to the boundary.  

When the sign of $F_u$ is not necessarily positive, Caffarelli, Nirenberg and Spruck \cite{CNS} solved the Dirichlet problem \eqref{eq: MA-equation} under the assumption of the existence of a {\em subsolution}. Although their result is very natural from the PDE perspective, often constructing such a subsolution is a difficulty task, one that is almost as difficult as constructing a solution itself. A different approach without constructing a subsolution was taken by Tso \cite{Tso}, he used a variational approach to study such problems using the Monge-Amp\`ere functional introduced by Bakelman \cite{Bakelman, Bakelman2}. In particular, Tso studied the family of Dirichlet problems
\begin{equation}
\begin{cases}
\det D^2u  = (-u)^p&\text{ in }\Omega\\
u = 0&\text{ on }\partial \Omega,
\end{cases}
\end{equation}
and he proved the existence and uniqueness of solutions when $p< n$ (subcritical case), and the existence of nontrivial solutions when $p>n$ (supercritical case). The case $p=n$ is the Monge-Amp\`ere eigenvalue problem and has been studied in earlier work of Lions \cite{Lions}. The variational approach turns out to be very natural and has been extended to the study of more general Hessian equations in \cite{Wang, Chou-Wang}. 

One feature of using the classical Monge-Amp\`ere functional is that it works well when $F$ depends on $u$ and not on its gradient $\nabla u$. However, for more general $F$ that depends on both $u$ and $\nabla u$, one does not always expect to have a variational formulation. In this paper, we study the solvability of a class of Dirichlet problems whose right-hand side depend on $u$ and on its gradient $\nabla u$ in terms of its Legendre tranform $u^{\star} = \langle x, \nabla u\rangle-u$. Such equations arises naturally from the study of affine spheres \cite{Chen-Huang, Klartag}, and similar equations also arises in the recent work of Collins-Li \cite{Collins-Li} on a generalization of a construction of Tian-Yau for complete Calabi-Yau metrics. We discover a variational structure for this class of equations, and use this to undertake a variational study of such equations. In particular, we define a family of {\em generalized Monge-Amp\`ere functionals} $H_{n+k}$ for any $k\in\mathbb R$, whose variations are given by
\[\delta H_{n+k}(u)[\phi] = - \int_\Omega \phi (u^{\star})^k\det D^2u .\]
The functional themselves are of independent interest, and one of them in particular suggests a close relationship to the study of convex geometry (see Theorem~\ref{thm: -n-1 case}). 

An important ingredient in our variational approach is a {\em Sobolev type inequality} (Theorem~\ref{thm: sobolev-inequality}) for $H_{n+k}$, which says that the functionals $H_{n+k}$ bound the $L^{n+k+1}$ norm of a convex function. The complex analogue of such Sobolev type inequalities and the limiting Trudinger type inequalities are of great importance in the study of complex Monge-Ampere equations, and sharp forms of the limiting Trudinger type inequalities for the complex Monge-Ampere equation have been established and used by Guo, Phong and the first author to prove sharp estimates for complex Monge-Ampere equations on compact K\"ahler manifolds \cite{Guo-Phong-Tong}. In this paper, we will use the Sobolev inequality for $H_{n+k}$ to study the solvability of the family of Dirichlet problems
\begin{equation}\label{eq: main-equations}
\begin{cases}
	\det D^2u  = \lambda(u^{\star})^{-k}(-u)^p&\text{ in }\Omega\\
	u = 0&\text{ on }\partial \Omega. 
\end{cases}
\end{equation}
The cases $p<n+k$ (subcritical), and $p>n+k$ (supercritical) are treated seperately, and the case $p=n+k$ is treated as an eigenvalue problem. Our main theorem is the following
\begin{theorem}\label{thm: main-thm}
	Let $n+k>0$ and consider the Dirichlet problem \eqref{eq: main-equations}. Then
	\begin{enumerate}
		\item If $0<p<n+k$, then for any $\lambda>0$,  \eqref{eq: main-equations} admits a unique nontrivial solution $u\in C^{2, \alpha}(\overline\Omega)\cap C^{\infty}(\Omega)$. 
		\item If $p= n+k$, then there exists a unique $\lambda$ for which \eqref{eq: main-equations} admits a non-trivial solution $u\in C^{2, \alpha}(\overline\Omega)\cap C^{\infty}(\Omega)$. Moreover, the solution is unique up to scaling by a positive constant. 
		\item If $p> n+k$, then for any $\lambda>0$, \eqref{eq: main-equations} admits a nontrivial solution $u\in C^{2, \alpha}(\overline\Omega)\cap C^{\infty}(\Omega)$. 
	\end{enumerate}
\end{theorem}

This paper is organized as follows. In Section~\ref{section: functionals}, we will define the functionals $H_{n+k}$ and establish their basic properties. In Section~\ref{section: gradient flow}, we study a parabolic gradient flow of functionals associated with $H_{n+k}$ and establish some basic apriori estimates for these equations. In section~\ref{section: dirichlet problems}, we will use the parabolic gradient flow to prove convergence results and establish existence and uniqueness for the associated Dirichlet problems. In particular, we will prove Theorem~\ref{thm: main-thm}.  Finally in section~\ref{section: tranformation}, we discuss a transformation which transforms the solution of the Dirichlet problem to a solution of an optimal transport problem and discuss some relations to other works. 
\section{The generalized Monge-Amp\`ere functionals}\label{section: functionals}
Let us fix $ \Omega\subset \mathbb R^n$ to be an open, bounded, smooth, and strictly convex domain in $\mathbb R^n$, and we will always assume the origin is contained in $\Omega$. Let $u\in C^\infty(\overline\Omega)$ be a strictly convex function on $\overline\Omega$ which vanishes on the boundary, and we denote $\mathcal C_0$ to be the space of all such functions, 
\[\mathcal C_0 = \{u\in C^\infty(\overline\Omega): D^2u>0, u|_{\partial\Omega} = 0\}. \]
For $u\in \mathcal C_0$, we denote $u^{\star}$ to be the function
\[u^{\star}(x):=\langle x, \nabla u(x)\rangle-u(x),\] 
which is the legendre transform of $u$ evaluated at the point $\nabla u(x)$. We note that $u^{\star}$ is nonnegative since $\inf_{x\in \Omega}u^{\star}(x) = u^{\star}(0) = -u(0)\geq 0$, and $u^{\star}$ is strictly positive unless $u \equiv 0$.

For any number $k\in \mathbb R$, we can define a {\bf generalized Monge-Amp\`ere functional} $H_{n+k}$ by
\begin{equation}\label{def: functional-H}
H_{n+k}(u) := \frac{1}{n+k+1}\int_{\Omega}(u^{\star})^{k}(-u)\det D^2u. 
\end{equation}
\begin{remark}
	Notice that $H_{n+k}(u)$ is always positive for $u\in \mathcal C_0$ since $u^{\star}\geq |u(0)|>0$ for $u\in \mathcal C_0$, and it vanishes only when $u$ is identically zero. 
\end{remark}

When $k=0$, $H_n$ reduces to the classical Monge-Amp\`ere functional first introduced by Bakelman \cite{Bakelman, Bakelman2}, and has been widely studied \cite{Cheng-Yau, Tso}. Other cases of interest include $k = -n-2$, which is related to the equation for  an elliptic affine spheres \cite{Calabi, Klartag}, and the case of $k\geq 0$ is closely related to a class of equations studied in \cite{Chen-Huang}, arising from the study of hyperbolic affine spheres. We shall see below that the scale-invariant case of $k = -n-1$ will turn out to be very special. 

\subsection{Variational formulas}
We compute the first and second variations of $H_{n+k}$. The first variation of $H_{n+k}$ with respect to $\phi\in C^\infty(\overline\Omega)$ satisfying $\phi|_{\partial \Omega}=0$ is given by

\begin{align}\label{eq: variation-of-H}
\delta H_{n+k}(u)[\phi]& = -\frac{1}{n+k+1}\int_{\Omega}\phi(u^{\star})^k\det D^2u+\frac{k}{n+k+1}\int_\Omega (\langle x, \nabla \phi\rangle-\phi)(u^{\star})^{k-1}(-u)\det D^2u\\
& \qquad +\frac{1}{n+k+1}\int_\Omega (u^{\star})^k(-u)u^{ij}\phi_{ij}\det D^2 u\\
& =  -\frac{1}{n+k+1}\int_{\Omega}\phi(u^{\star})^k\det D^2u+\frac{k}{n+k+1}\int_\Omega (\langle x, \nabla \phi\rangle-\phi)(u^{\star})^{k-1}(-u)\det D^2u\\
& \qquad -\frac{k}{n+k+1}\int_\Omega (u^{\star})^{k-1}(-u)\langle x, \nabla \phi\rangle\det D^2 u+\frac{1}{n+k+1}\int_\Omega (u^{\star})^ku^{ij}\phi_iu_j\det D^2u\\
& =  -\frac{1}{n+k+1}\int_{\Omega}\phi(u^{\star})^k\det D^2u-\frac{k}{n+k+1}\int_\Omega \phi(u^{\star})^{k-1}(-u)\det D^2u\\
& \qquad -\frac{n}{n+k+1}\int_\Omega \phi(u^{\star})^k\det D^2u-\frac{k}{n+k+1}\int_\Omega \phi(u^{\star})^{k-1}\langle x, \nabla u\rangle\det D^2u\\
& = -\int_\Omega\phi (u^{\star})^k\det D^2u\,dx
\end{align}

The second variation with respect to $\phi, \psi\in C^\infty(\overline\Omega)$, satisfying $\phi|_{\partial\Omega}=\psi|_{\partial\Omega}=0$ is given by
\begin{align}\label{eq: second-variation-H}
\delta^2 H_{n+k}(u)[\phi, \psi]
& =\int_\Omega u^{ij}\phi_i\psi_j(u^{\star})^k\det D^2u 
+k\int_\Omega \phi\psi(u^{\star})^{k-1}\det D^2 u. 
\end{align}
From this, we can see that if $k\geq 0$, then this norm is always positive definite, hence $H_{n+k}$ is convex on $\mathcal C_0$. Using the functionals $H_{n+k}$, we may define a norm $\|\cdot\|_{H_{n+k}}$ on the space $\mathcal C_0$ by setting
\begin{equation}\label{def: H-norm}
\|u\|_{H_{n+k}} := [H_{n+k}(u)]^{\frac{1}{n+k+1}}
\end{equation}
\begin{prop}
	If $k\geq 0$, then for $u, v\in \mathcal C_0$, the triangle inequality holds
	\[\|u+v\|_{H_{n+k}} \leq \|u\|_{H_{n+k}}+\|v\|_{H_{n+k}}.\]
\end{prop}
\begin{proof}
	By applying integration by parts to formulas \eqref{def: functional-H} and \eqref{eq: variation-of-H}, we  obtain
	\[H_{n+k}(u) = \frac{1}{(n+k)(n+k+1)}\left(\int_\Omega u^{ij}u_iu_j(u^{\star})^{k}\det D^2u+k\int_\Omega(-u)^2(u^{\star})^{k-1}\det D^2u\right)\]
	and
	\[\delta H_{n+k}(u)[\phi] = \frac{1}{n+k}\left(\int_\Omega u^{ij}u^i\phi_j(u^{\star})^{k}\det D^2u+k\int_\Omega(-\phi)(-u)(u^{\star})^{k-1}\det D^2u\right)\]
	and we also have
	\[\delta^2 H_{n+k}(u)[\phi, \psi] 
	=\int_\Omega u^{ij}\phi_i\psi_j(u^{\star})^k\det D^2u\,dx 
	+k\int_\Omega \phi\psi(u^{\star})^{k-1}\det D^2 u\,dx. \]
	Therefore by the Cauchy Schwarz inequality, we obtain
	\[\delta^2H_{n+k}(u)[v, v]\cdot H_{n+k}(u)\geq \frac{n+k}{n+k+1}\left(\delta H(u)[v]\right)^2,\]
	from which the triangle inequality follows. 
\end{proof}
\subsection{Sobolev type inequalities}
The functional $H_{-1}$ (suitably normalized) turns out to be very special. In fact, it turns out to be independent of $u$. 
\begin{theorem}\label{thm: -n-1 case}
    For any $u\in \mathcal C_0$, we have
    \[\int_\Omega \frac{-u\det D^2u}{(u^{\star})^{n+1}} = |\Omega^{\circ}|, \]
    where $\Omega^{\circ}$ the {\bf polar body} of $\Omega$, defined by
    \[\Omega^{\circ} := \{x: \langle x, y\rangle\leq 1 \text{ for all }y\in \Omega\}.\]
\end{theorem}
\begin{proof}
By the variational formula from above, we have
\[\delta \left(\int_\Omega \frac{-u\det D^2u}{(u^{\star})^{n+1}}\right) \equiv 0\]
which implies 
\[\int_\Omega \frac{-u\det D^2u}{(u^{\star})^{n+1}} = const. \]
It suffices to show that this constant is equal to $|\Omega^{\circ}|$. Let $\hat u$ be the function 
\[\hat u(x) = \inf\left\{\frac{1-t}{t}: t>0, tx\in\Omega\right\},\] then it's easy to check that $\hat u\in C^{\infty}(\Omega\setminus \{0\})$ and $\det D^2\hat u = 0$ on $\Omega\setminus \{0\}$. Let $\phi_{\eps}:\mathbb R_{\geq 0}\to \mathbb R_{\geq 0}$ be a smooth function satisfying $\phi_{\eps}(x) = x$ for $x\in [0, 1-2\eps]$,  $\phi_{\eps}(x) = 1-\eps$ in a neighborhood of $1$, $\phi_{\eps}'>0$ and $\phi_{\eps}''<0$. Then define $u_{\eps}(x) = -\phi_{\eps}(-\hat u(x))+\eps\rho(x)$, where $\rho$ is a strictly convex defining function for $\Omega$, then one can check that the following are true,
\begin{enumerate}
	\item $\|u_{\eps}^{\star}-1\|_{L^{\infty}(\Omega)} = O(\eps)$. 
	\item $\nabla u_{\eps}(\Omega)$ is decreasing to $\Omega^{\circ}$ as $\eps\to 0$. 
	\item For $y\in \Omega^{\circ}$, let $x_{\eps}\in \Omega$ be the point such that $\nabla u_{\eps}(x_{\eps}) = y$, then  $\lim_{\eps\to 0}u_{\eps}(x_{\eps}) = -1$. 
\end{enumerate}
By a change of variables formula, we have
\begin{align}\int_\Omega\frac{-u_{\eps}\det D^2u_{\eps}}{(u_{\eps}^{\star})^{n+1}} 
&= \int_{\nabla u_{\eps}(\Omega)}\frac{-u_{\eps}(x_{\eps})}{(u_{\eps}^{\star}(x_{\eps}))^{n+1}}dy \\
&= \int_{\Omega^{\circ}}\frac{-u_{\eps}(x_{\eps})}{(u_{\eps}^{\star}(x_{\eps}))^{n+1}}dy +\int_{\nabla u_{\eps}(\Omega)\setminus \Omega^{\circ}}\frac{-u_{\eps}(x_\eps)}{(u_{\eps}^{\star}(x_{\eps}))^{n+1}}dy \end{align}
where $y\in \nabla u_{\eps}(\Omega)$ and $x_{\eps}\in\Omega$ are related by $\nabla u_{\eps}(x_{\eps}) = y$. The second integral converges to zero as $\eps\to 0$ since $\nabla u_{\eps}(\Omega)$ converges to $\Omega^{\circ}$. For the first term, note that both $-u_{\eps}(x_{\eps})$ and $u_{\eps}^{\star}(x_\eps)$ converges to $1$ for any $y\in \Omega^{\circ}$, therefore the first integral converges to $|\Omega^{\circ}|$ and we've proven the claim. 
\end{proof}
From Theorem~\ref{thm: -n-1 case}, we can deduce a Sobolev type inequality for $H_{n+k}$. 
\begin{theorem}\label{thm: sobolev-inequality}
    Suppose $n+k+1>0$, then there exist $c>0$ such that for any $u\in \mathcal C_0$, we have
\begin{equation}\label{ineq: sobolev-ineq}
        H_{n+k}(u)\geq c\int_\Omega |u|^{n+k+1}. 
    \end{equation}
    Moreover, $c$ depends only on $n, k$ and $\Omega$. 
\end{theorem}
\begin{proof}
By Theorem~\ref{thm: -n-1 case} and the fact that $u^{\star}(x)\geq |u(0)|$, we obtain the inequality, 
\begin{equation}\label{ineq: H-lower-bdd}
H_{n+k}(u)\geq \frac{|\Omega^{\circ}|}{n+k+1}|u(0)|^{n+k+1}. 
\end{equation}
Since $0\in \Omega$, by the convexity of $u$, we have $|u(0)|\geq \frac{d(0, \partial\Omega)}{diam(\Omega))}\|u\|_{L^{\infty}}$, and $\|u\|_{L^{\infty}}^{n+k+1}\geq c\int_\Omega |u|^{n+k+1}$. Together with \eqref{ineq: H-lower-bdd}, this proves the claim. 
\end{proof}

Using this, we can define an invariant of $\Omega$ to be the maximum of all such $c$ such that inequality~\eqref{ineq: sobolev-ineq} holds. 
\begin{equation}\label{eq: def-of-lambda-n+k+1}
    \underline\lambda_{n+k+1}(\Omega) :=\inf_{u\in \mathcal C_0}\left\{\frac{(n+k+1)H_{n+k}(u)}{\int_\Omega|u|^{n+k+1}}\right\}. 
\end{equation}
In the case $k=0$, $\underline\lambda_{n+1}$ was studied by Tso \cite{Tso}, who showed that it is essentially equal to the Monge-Amp\`ere eigenvalue introduced by Lions \cite{Lions}. We will show in section~\ref{section: dirichlet problems}, that each $\underline{\lambda}_{n+k+1}$ will correspond to the solution of a nonlinear eigenvalue problem. 

\begin{remark}
	The discussion above can be generalized to allow for more general dependence on $u^{\star}$. For a positive function $h:\mathbb R_+\to \mathbb R_+$, we denote $G(x) :=\int_0^xs^nh(s)ds$, then we can consider the functional
	\[H_h(u):= \int_{\Omega}G(u^{\star})\frac{(-u)\det D^2u}{(u^{\star})^{n+1}}. \]
	Then its variation will be given by 
	\[\delta H_h(u)[\phi] = -\int_\Omega\phi h(u^{\star})\det D^2u, \]
	and second variation is
	\[\delta^2 H_h(u)[\phi, \psi] = \int_\Omega\phi\psi h'(u^{\star})\det D^2u+\int_\Omega u^{ij}\phi_i\psi_jh(u^{\star})\det D^2u. \]
	Many of the results of this section holds for more general $h$. 
\end{remark}

\section{A parabolic gradient flow}\label{section: gradient flow}
Let $F(x, u):\overline\Omega\times \mathbb R_{\leq 0}\to \mathbb R_{\geq 0}$ be a non-negative function. In this section, we will study a gradient flow of the functional
\begin{equation}\label{def: J-functional}
\mathcal J(u) := H_{n+k}(u)-\int_{\Omega}\left(\int_u^0F(x, s)ds\right).
\end{equation}
whose variation is
\begin{equation}\label{eq: var-of-J}
\delta \mathcal J(u)[\phi] = -\int_\Omega \phi\left((u^{\star})^k\det D^2u-F(x, u)\right). 
\end{equation}

First to establish some notation, we denote $Q_T := \Omega\times (0, T]$ to be the parabolic cylinder, and $\overline Q_T = \overline\Omega\times [0, T]$ is the closure of $Q_T$ in $\mathbb R^n\times \mathbb R_{\geq 0}$. Let $\partial^{\star}\Omega := \overline Q_T\setminus Q_T$ denote the parabolic boundary, and $\Sigma := \partial\Omega\times (0, T)$ be the spacial boundary. We will define the parabolic H\"older norms
\[\|u\|_{\tilde C^{k, \alpha}(\overline Q_T)} := \sum_{2m+l\leq k}\sup_\Omega |D_t^mD_x^l u|+\sum_{2m+l=k}\sup_{x\neq y} \frac{|D_t^mD_x^l u(x)-D_t^mD_x^l u(y)|}{|x-y|^\alpha},\]
and denote $\tilde{C}^{k, \alpha}$ to be space of all functions with finite $\tilde C^{k, \alpha}$-norm. 

Let us assume for the rest of this section that $F(x, u)\geq \eta>0$ is strictly positive and denote \[g(x, u) :=\log F(x, u),\] and consider the following parabolic equation with initial condition $u_0\in \mathcal C_0$. 
\begin{equation}\label{eq: parabolic-equation}
\begin{cases}
u_t-\log\det D^2u-k\log (u^{\star}) = -g(x, u) &\text{ in }Q_T\\
u = u_0 &\text{ on } \overline\Omega\times \{0\} \\
u = 0 &\text{ on }\partial \Omega\times (0, T).  \end{cases}
\end{equation}
By the first variation formula \eqref{eq: var-of-J} of $\mathcal J$, we can see that $\mathcal J$ is non-increasing along this flow. 

In order to obtain the existence of solutions to \eqref{eq: parabolic-equation}, we will first consider the slightly more general boundary value problem
\begin{equation}\label{eq: parabolic-equation-with-boundary}
\begin{cases}
u_t-\log\det D^2u-k\log (u^{\star}) = -g(x, u) &\text{ in }Q_T\\
u = \Phi &\text{ on } \partial^{\star}Q_T   \end{cases}
\end{equation}
where we assume that the boundary data $\Phi$ can be extended to $\Phi\in \tilde C^2(\overline Q_T)$, such that $\Phi|_{\{0\}\times \overline\Omega} = u_0$ and $u_0(0)<0$.  
We assume moreover that $\Phi$ satisfy the compatibility condition 
\begin{equation}\label{eq: compatibility condition}
\Phi_t = \log \det D^2u_0+k\log u_0^{\star}-g(x, u_0) \text{ on } \{0\}\times \partial \Omega. 
\end{equation}

From now on, we will assume $n+k\geq 0$. The following theorem establishes a priori estimates for the parabolic equation \eqref{eq: parabolic-equation-with-boundary}. 
\begin{theorem}\label{thm: estimate-parabolic}
	Assume $n+k\geq 0$. Let $u\in \tilde C^4(Q_T)\cap \tilde C^2(\overline Q_T)$ satisfy equation~\eqref{eq: parabolic-equation-with-boundary}, and suppose further that
	\[u(t, 0)\leq-\eps<0 \text{ and } \|u\|_{C^0}\leq K.\]
	Then we have the estimate
	\[\|u\|_{\tilde C^2(Q_T)}\leq C\]
	for $C$ depending on $n, k, \Omega$, the initial condition $u_0$, $\eps$, $K$, $\|g\|_{C^2(\overline\Omega \times [-K, 0])}$, $\|\Phi\|_{C^k(\Sigma)}$, and $\|D_t\delta\Phi\|_{C^0(\Sigma)}$
\end{theorem}
\begin{proof}
The proof of these estimates are by now standard, and we only give a sketch of the proof here. The proof proceeds by first estimating $|u_t|$, then estimating $|\nabla u|$, and finally $|\nabla ^2 u|$. The estimate for $u_t$ follows from the same arguments as in \cite[Appendix]{Tso}. The estimate for $\nabla u$ follows from the construction of a subsolution from \cite[pg 391]{CNS}. The estimates for $\nabla^2u$ follow from combining the arguments in \cite[Appendix]{Tso} with the arguments of \cite[Section 7]{CNS} to handle the extra gradient terms. 
\end{proof}

Once we have this, we can obtain higher regularity for $u$ on a forward interval of time by \cite{Krylov2} and a standard bootstrap argument using Schauder estimates. 
\begin{corr}
For any $\delta>0$, we have higher order estimates for $u$
\[\|u\|_{\tilde C^{l, \alpha}(\overline\Omega\times [\delta, T])}\leq C, \]
where $C$ depends on $n, k, l, \alpha, \delta$, $K$, $\eps$ and bounds for higher derivatives of $g$ and higher derivatives of $\Phi$ on $\partial\Omega\times [\delta, T]$, but is independent of $T$. 
\end{corr}

The estimates in Theorem~\ref{thm: estimate-parabolic} can be used to obtain short-time solution to the Dirichlet problems \eqref{eq: parabolic-equation}. 

\begin{prop}[Short-time existence]\label{prop: short-time-existence}
	Assume $n+k\geq 0$. Then for any initial $u_0\in \mathcal C_0$, there exist a unique maximal time $T\in (0, \infty]$ and a unique solution $u(t, x)\in \tilde C^{1, 1}(\overline Q_T)\cap C^{\infty}(\overline\Omega\times (0, T])$ of \eqref{eq: parabolic-equation} in $Q_T$, such that
	\begin{enumerate}
		\item Either $T = \infty$ or $\limsup_{t\to T}\|u(t, \cdot)\| = \infty$. 
		\item For any $t>0$, we have
		\[\mathcal J(u(t, \cdot))\leq  \mathcal J(u_0). \]
	\end{enumerate}
\end{prop}
\begin{proof}
	Short-time existence follows from the estimates of Theorem~\ref{thm: estimate-parabolic} by standard arguments (see \cite[Theorem A]{Tso}).  Suppose now that $T$ is a finite time singularity, and let $r$ be small enough such that $B_r(0)\subset \Omega$. Since $F$ is strictly positive, there exist $\eta>0$ such that $F(x, u)\geq \eta>0$. It follows that for sufficiently small $\eps$, the function \[v = \eps (|x|^2-r^2)\] is a supersolution of \eqref{eq: elliptic-equations-F} on $B_r(0)$. By the comparison principle, we have $u(t, 0)\leq v(0)<0$. Then the estimates of Theorem~\ref{thm: estimate-parabolic} show that if $T$ is a singular time, we must have
	\[\limsup _{t\to T}\|u(t, \cdot)\| = \infty. \]
	The second part follows from the same argument as in \cite[Theorem A]{Tso}.
\end{proof}

\section{Associated Dirichlet problems}\label{section: dirichlet problems}
In this section, we use a variational method to study the solvability of the family of Dirichlet problems
 \begin{equation}
\begin{cases}\label{eq: degen-dirichlet-problem}
(u^{\star})^k\det D^2u = |u|^{p} &\text{ in }\Omega\\
u=0 &\text{ on }\partial\Omega.
\end{cases}
\end{equation}
 The case $p<n+k$, $p>n+k$, $p=n+k$ will be treated seperately. We will refer to the three different cases as the {\em subcritical}, {\em supercritical} and {\em critical} case respectively. 

More generally, let $F(x, u)\geq 0$ be a non-negative function such that $F(x, u)>0$ for $u<0$. Then we will consider Dirichlet problems of the form
\begin{equation}\label{eq: elliptic-equations-F}
    \begin{cases}
        (u^{\star})^k\det D^2u = F(x, u) &\text{ in }\Omega\\
        u=0 &\text{ on }\partial\Omega,
    \end{cases}
\end{equation} 
whose solutions are the critical points of $\mathcal J$, and are the stationary points of the flow \eqref{eq: parabolic-equation}.

We will need the following estimates for solutions of \eqref{eq: elliptic-equations-F}, which by now are standard. 

\begin{theorem}\label{thm: elliptic estimates}
	Let $n+k\geq 0$ and $u$ be a solution to \eqref{eq: elliptic-equations-F} with 
	\[\|u\|_{L^{\infty}}\leq K\]
	and 
	\[u(0)\leq -\eps<0. \]
	 Then we have the estimates
	\begin{enumerate}
		\item We have a gradient estimate \[\sup_\Omega|\nabla u|\leq C\]
		for some $C$ depending on $n, k, \eps, \Omega, K$, and,  $\sup_{\Omega\times [-K, 0]} F$. 
		\item For every $\Omega'\ssubset \Omega$, we have the estimates
		\[\|u\|_{C^{l, \alpha}(\Omega')}\leq C\]
		for $C$ depending on $n, k, \Omega, d(\Omega', \partial\Omega), l, \alpha, K, \eps$, a positive lower bound for $F$ on $\Omega'$, and bounds on $F$ and its derivatives on $\Omega'\times [-K, \sup_{\Omega'} u]$. 
		\item If in addition $F(x, u)\geq \eta >0$ is strictly positive. Then we have
		\[\|u\|_{C^{l, \alpha}(\overline\Omega)}\leq C\]
		for $C$ depending on $n, k, \Omega, l, \alpha, K, \eps, \eta$, and $\|F\|_{C^{l+5}(\overline\Omega\times[-K, 0])}$.
	\end{enumerate}
\end{theorem}
\begin{proof}
	The estimate for $|\nabla u|$ can be deduced from the subsolution construction following the argument preceeding \cite[Theorem 7]{CNS}. The second statement follows from Pogorelov's interior estimates \cite{Pogorelov} (see also \cite{Gilbarg-Tru}) and Evans-Krylov \cite{Evans, Krylov}. The third part follows from \cite[Theorem 7]{CNS} or \cite{Krylov}. 
\end{proof}

\subsection{Subcritical case}
Now we will turn our attention to the study the Dirichlet problem \eqref{eq: elliptic-equations-F}. The first result deals with the subcritical case. 

\begin{theorem}[Subcritical case]\label{thm: subcritical}
Suppose that $F(x, u)\geq \eta>0$ satisfies 
\begin{enumerate}
    \item There exist $\lambda< \underline\lambda_{n+k+1}$, and $M>1$ such that $\int_u^0F(x, s)ds\leq \lambda|u|^{n+k+1}$ for $|u|\geq M$. 
    \item $\inf_{u\in \mathcal C_0}\mathcal J(u)<\mathcal J(0) = 0$
\end{enumerate}
then \eqref{eq: elliptic-equations-F} admit a solution. 
\end{theorem}

\begin{proof}
By assumption~1 and the definition of $\underline\lambda_{n+k+1}$ \eqref{eq: def-of-lambda-n+k+1}, we have
\begin{equation}\label{eq: subcri-lower-bound-J}
\mathcal J(u)\geq \underline\lambda_{n+k+1}\int_\Omega |u|^{n+k+1}-\int_\Omega\left|\int_u^0F(x, s)ds\right|\geq c\|u\|_{L^{\infty}}^{n+k+1}-CM,
\end{equation}
which implies $\mathcal J(u)$ is bounded from below. Since $\inf_{u\in \mathcal C_0}J(u)<\mathcal J(0)$, and $\mathcal J(u)$, we can pick $u_0\in \mathcal C_0$ so that $\mathcal J(u_0)<0$. Let $u(t, x)$ be the solution of the parabolic equation \eqref{eq: parabolic-equation}, then by Proposition~\ref{prop: short-time-existence}, we have $\mathcal J(u(t, \cdot))<\mathcal J(u_0)< 0$. We claim that there exist $K_0, K_1>0$ such that 
\[K_0^{-1}\leq \|u(t, \cdot)\|_{L^{\infty}}\leq K_1. \]
The upper bound is a consequence of \eqref{eq: subcri-lower-bound-J} and the fact that $\mathcal J(u(t, \cdot))$ is decreasing. The lower bound follows from the fact that $\mathcal J(u(t, \cdot))$ is negative, therefore there is some $\eps>0$ such that $\mathcal J(u(t, \cdot))<-\eps<0$, which means
\[\int_X \int_{u(t, x)}^0F(x, s)ds\geq \eps. \]
and it follows from Chebychev's that $\|u\|_{L^{\infty}}$ must be bounded from below. Once we have a uniform bound of $\|u(t, \cdot)\|_{L^{\infty}}$ from above and below, it follows by Proposition~\ref{prop: short-time-existence} that $u(t, x)$ exists for all time, and the estimates in Theorem~\ref{thm: estimate-parabolic} gives us bounds of all order on $u(t, \cdot)$ as $t\to\infty$. Since $\mathcal J(u(t, \cdot))$ is decreasing and bounded below, we can extract a subsequence of times $t_j\to \infty$ for which
\[\lim_{j\to \infty}\left.\frac{d}{dt}\right|_{t = t_j} \mathcal J(u(t, \cdot)) = 0.\]
If we denote $u_j(x) = u(t_j, x)$, then we can rewrite this as
\[\lim_{j\to \infty}\int_\Omega\left(\log ((u_j^{\star})^k\det D^2u_j)-\log F(x, u_j)\right)\left((u_j^{\star})^k\det D^2u_j-F(x, u_j)\right) = 0. \]
From this and the estimates on $\|u(t_j, \cdot)\|_{C^{k, \alpha}}$, we can extract a convergent subsequence $u(t_j, \cdot)\to u_{\infty}$ which must be a solution of the equation \eqref{eq: elliptic-equations-F}. 
\end{proof}

Now we prove a uniqueness results. For the uniquess theorem, we do not need to assume $F$ is strictly positive. 
\begin{theorem}[Uniqueness]\label{thm: uniqueness}
Suppose $n+k\geq 0$ and $F(x,u)\geq 0$ satisfies 
\begin{enumerate}
    \item $F(x, u)>0$ for $u<0$. 
    \item $F(x, tu)\geq t^{n+k}F(x, u)$ for all $0<t<1$. 
    \item For $0<|u|\ll 1$, the inequality above is strict. 
\end{enumerate}
Then there exist at most one non-zero solution to \eqref{eq: elliptic-equations-F} in $C^{2, \alpha}(\overline\Omega)$. 
\end{theorem}
\begin{proof}
Assume there are two distict non-zero solutions $u_1$ and $u_2$ to \eqref{eq: elliptic-equations-F}. Without loss of generality we may assume that $|u_1| \nleq |u_2|$. Let 
\[t_0 := \sup \{t>0: t|u_1|\leq |u_2|\}. \]
then $t_0<1$ and let us set $\tilde{u}_1 = t_0u_1$, then by assumption 2, we have
\[(\tilde u^{\star}_1)^k\det D^2\tilde u_1 = t_0^{n+k}F(x, u_1)\leq F(x, \tilde u_1),\]
which implies 
\begin{align}
0&\leq (\det D^2u_2)^{\frac 1 n}-(\det D^2\tilde u_1)^{\frac 1 n}
-\left(\frac{F(x, u_2)^{\frac{1}{n}}}{(u_2^{\star})^{\frac{k}{n}}}-\frac{F(x, \tilde u_1)^{\frac{1}{n}}}{(\tilde u_1^{\star})^{\frac{k}{n}}}\right) \\
&\label{eq: subsol-(u-v)} \leq \frac{1}{n}(\det D^2\tilde u_1)^{\frac{1}{n}}\tilde u_1^{ij}(u_2-\tilde u_1)_{ij}
-\left(\frac{F(x, u_2)^{\frac{1}{n}}}{(u_2^{\star})^{\frac{k}{n}}}-\frac{F(x, \tilde u_1)^{\frac{1}{n}}}{(\tilde u_1^{\star})^{\frac{k}{n}}}\right) 
\end{align}
where the second line follows from the concavity of $M \mapsto (\det M)^{\frac 1 n}$. 
By our choice of $t_0$, we have $|\tilde u_1|\leq |u_2|$ and one of the following must be true. \begin{enumerate}
    \item $\tilde u_1(x_0) = u_2(x_0)$ for some $x_0\in \Omega$
    \item $\langle x_0, D\tilde u_1(x_0)\rangle = \langle x_0, Du_2(x_0)  \rangle $ for some $x_0\in \partial\Omega$. 
\end{enumerate}
In the first case, we have $\tilde u_1\equiv u_2$ by the strong maximum principle, and in the second case we need to apply the Hopf lemma. If $\det D^2\tilde u_1(x_0)>0$, then we can apply the standard Hopf lemma \cite[Lemma 3.4]{Gilbarg-Tru}. In the case $\det D^2\tilde u_1(x_0)=0$, we will need the following elementary lemma in order to apply the Hopf lemma. 
\begin{lem}
	Let $x_0\in \partial \Omega$ be a boundary point, and $B_r(y)\subset \Omega$ be a interior ball that touches $\partial \Omega$ at $x_0$. Then on $B_r(y)\setminus B_{r/2}(y)$, there exist a constant $c>0$ such that
	\begin{equation}\label{eq: normal-control-inverse}
	\tilde u_1^{ij}(x-y)_i(x-y)_j\geq c|(D^2\tilde u_1)^{-1}|.
	\end{equation}
\end{lem}
\begin{proof}
	This follows from the fact that $\tilde u_1\in C^{2, \alpha}$ and $D^2\tilde u_1(x_0)$ is strictly positive in the tangential direction to $\partial\Omega$. Let us fix coordinate $(\tilde x_1, \ldots, \tilde x_n)$ centered at $x_0$ such that locally $\partial\Omega$ is the graph of a function $\rho(\tilde x_1, \ldots, \tilde x_{n-1})$ and $\rho(\tilde x') = \frac{1}{2}\sum_{i =1}^{n-1}\rho_i\tilde x_i^2 + O(|\tilde x'|^3)$ where $\rho_i>0$ are the principal curvatures of $\partial \Omega$ at $x_0$. By the boundary condition, we have for $i, j\in \{1, \ldots, n-1\}$
	\[(\tilde u_1)_{ij}(x_0) = -(\tilde u_1)_n(x_0)\delta_{ij}\rho_i, \]
	and since $\tilde u_1\in C^{2, \alpha}$, we have
	\[D^2\tilde u_1(x) = \begin{pmatrix}
	-(\tilde u_1)_n(x_0) \rho_1& 0 & 0& (\tilde u_1)_{n, 1}(x_0)\\ 
	0 & \ddots &0&\vdots  \\ 
	0&  0&-(\tilde u_1)_n(x_0)\rho_{n-1}& (\tilde u_1)_{n, n-1}(x_0) \\ 
	(\tilde u_1)_{1, n}(x_0)&\cdots &(\tilde u_1)_{n-1, n}(x_0)&(\tilde u_1)_{n, n}(x_0)
	\end{pmatrix}+O(|x-x_0|^{\alpha}).\]
	Therefore for $|x-x_0|$ sufficiently small, we get the inequality \eqref{eq: normal-control-inverse}. When $|x-x_0|$ is large, the inequality \eqref{eq: normal-control-inverse} follow trivially and the lemma is proved.  
\end{proof}
This lemma will allow us to apply the Hopf lemma to the operator $L = \frac{1}{n}\tilde u_1^{ij}\partial_i\partial_j$. First we note that by \eqref{eq: subsol-(u-v)}, we have
\[(\det D^2\tilde u_1)^{1/n}L(u_2-\tilde u_1)\geq o_{|x-x_0|}(1)\]
where $o_{|x-x_0|}(1)$ goes to $0$ as $|x-x_0|$ goes to $0$.  If we fix coordinate system $(x_1, \ldots, x_n)$ such that $x_0 = (0, \ldots, 0)$ and $\Omega\subset \mathbb R^{n-1}\times \mathbb R_{\geq 0}$, then locally we can write $\partial\Omega$ as the graph of a function $x_n = \rho(x_1, \ldots, x_{n-1})$ and moreover 
\[\rho(x_1, \ldots, x_{n-1}) = \frac{1}{2}\sum_{i = 1}^{n-1}\rho_i|x_i|^2+O(|x'|^3),\]
where $\rho_i>0$ are the principal curvature of $\partial\Omega$ at the origin. Then we can pick $A$ sufficiently large such that in a small neighborhood $U = B_\delta(0)\cap\Omega$ of the origin, we have
 \[u_2-\tilde u_1+ (x_n^2-A\sum_{i=1}^{n-1}|x_i|^3)< 0\]  and using the lemma above, and by shrinking $U$ if necessary, we have that in $U$, 
\begin{align}
(\det D^2\tilde u_1)^{1/n}L(u_2-\tilde u_1+ (x_n^2-A|x'|^3))&\geq o_{|x|}(1)+(c+O(|x'|))(\det D^1\tilde u_1)^{\frac{1}{n}}\sum_{i = 1}^nu^{ii}\\
&\geq c+o_{|x|}(1)+O(|x'|)\\
&\geq 0
\end{align}

Now we fix a small interior ball $B_r(y)\subset U$ that touches $\partial\Omega$ at $x_0$ and consider the standard Hopf barrier $w = e^{-\alpha|x-y|^2}-e^{-\alpha r^2}$. By the standard calculation \cite[Lemma 3.4]{Gilbarg-Tru}, we have 
\[Lw\geq \alpha e^{-\alpha \frac{r}{4}^2}(4\alpha u_1^{ij}(x-y)_i(x-y)_j-2 \sum_i u_1^{ii}). \]
For $\alpha$ sufficiently large, the lemma above gives us $Lw\geq c|(D^2 \tilde u_1)^{-1}|\geq c(\det D^2\tilde u_1)^{\frac{-1}{n}}$, hence we have in $B_r(y)$
\[L(u_2-\tilde u_1+ (x_n^2-A|x'|^3)+\eps w)\geq 0  \]
and if we choose $\eps$ sufficiently small, then $u_2-\tilde u_1+ (x_n^2-A|x'|^3)+\eps w \leq 0$ on $B_r(y)\setminus B_{r/2}(y)$, and applying the maximum principle gives $u_2-\tilde u_1+ (x_n^2-A|x'|^3)\leq -\eps w$. In particular
\[-\frac{\partial (u_2-\tilde u_1)}{\partial x_n}(0)>\eps \frac{\partial w}{\partial x_n}(0)>0 \]
which is a contradiction. Hence we conclude $u_2 = t_0u_1$, which implies $t_0^{n+k}F(x, u_1) \equiv F(x, t_0u_1)$, which contradicts our third assumption on $F$. 
\end{proof}

Notice the assumptions in Theorem~\ref{thm: uniqueness} are satisfied if $F(x, u)^{\frac{1}{n+k}}$ is concave. In particular Theorem~\ref{thm: uniqueness} applies for $F(x, u) = |u|^p$ for $p<n+k$. As an corollary, we obtain the existence and uniqueness of solutions to the Dirichlet problem \eqref{eq: degen-dirichlet-problem} in the subcritical case.

\begin{corr}\label{cor: subcritical}
Let $0<p< n+k$, then the Dirichlet problems \eqref{eq: degen-dirichlet-problem} admits a unique non-zero convex solution $u\in C^{2, \alpha}(\overline\Omega)$, which minimizes the functional 
\[\mathcal J_p(u):= H_{n+k}(u)-\frac{1}{p+1}\int_{\Omega}(-u)^{p+1}. \]
\end{corr}
\begin{proof}
    By Theorem~\ref{thm: subcritical}, there exist a unique $u_{\eps}\in C^{\infty}(\overline\Omega)$ that solves the Dirichlet problem
    \[ \begin{cases}
    (u_{\eps}^{\star})^k\det D^2u_{\eps} = (\eps-u_{\eps})^p &\text{ in }\Omega\\
    u_{\eps} = 0 &\text{ on }\partial\Omega. 
    \end{cases} \]
    Moreover, $u_{\eps}$ is a global minimizer of 
    \[\tilde{\mathcal J}_{p, \eps}(w) = H_{n+k}(w)-\frac{1}{p+1}\int_{\Omega}(\eps-w)^{p+1}. \]
    It's clear that $u_{\eps}$ is uniformly bounded above and below, therefore by Theorem~\ref{thm: elliptic estimates}, we get interior estimates of all order and $u_{\eps}$ must converge along some subsequence to some non-zero $u\in C^{\infty}(\Omega)\cap C^{0, 1}(\overline\Omega)$ which is a solution of \eqref{eq: degen-dirichlet-problem}. By the arguments of \cite[Theorem 1.3]{Savin} and \cite[Theorem 1.2]{Le-Savin} we get that $u\in C^{2, \alpha}(\overline\Omega)$. 
    Moreover $u$ minimzes $\mathcal J_p$ because $u_{\eps}$ are minimzers of $\tilde{\mathcal J}_{\eps, p}$, and the uniqueness follows from Theorem~\ref{thm: uniqueness}. 
\end{proof}

\subsection{Supercritical case}
Now we treat the supercritical case. In this case, we will obtain non-trivial solutions to \eqref{eq: degen-dirichlet-problem} by a min-max method.

\begin{theorem}[Supercritical case]\label{thm: supercritical}
Assume that $n+k\geq 0$. Let $F(x, u)\geq \eta>0$ be strictly positive and $\mathcal J$ be the functional
\[\mathcal J(u) = H_{n+k}(u)-\int_\Omega\int_u^0F(x, s)\,ds. \]
Assume $F$ satisfy the following,
\begin{enumerate}
    \item There exist $c, \sigma>0$ such that 
    \[\mathcal J(u)\geq c\]
    for $\|u\|_{L^{\infty}} = \sigma$. 
    \item There exist $u_0$, $u_1\in \mathcal C_0$ such that $\|u_0\|_{L^{\infty}}<\sigma<\|u_1\|_{L^{\infty}}$, and $\mathcal J(u_0)<c$, and $\mathcal J(u_1) < c$. 
    \item There exist $\theta\in (0, 1)$, $M>0$, such that 
    \[\int_u^0F(x, s)ds\leq \frac{1-\theta}{n+k+1}|u|F(x, u)\]
    for all $|u|>M$. 
    \item $\left|\frac{F_u(x, u)}{F(x, u)}\right|\leq C$. 
\end{enumerate}
then \eqref{eq: elliptic-equations-F} admit a solution $u\in C^{\infty}(\overline\Omega)$ and $\mathcal J(u) \geq c$
\end{theorem}
\begin{proof}
The proof follows from a mountain-pass lemma, and is proved in the same way as \cite[Theorem C]{Tso} with some small modifications.  
Let
\[\mathcal P = \{\gamma:[0, 1]\to \mathcal C_0: \|\gamma(0)\|_{L^{\infty}}<\sigma<\|\gamma(1)\|_{L^{\infty}}, \mathcal J(\gamma(0))<c, \mathcal J(\gamma(1))<c\}\] 
and
\[d = \inf_{\gamma\in \mathcal P}\sup_{s\in [0, 1]}\mathcal J(\gamma(s))\geq c>0.\]
We will show that $d$ is a critical value of $\mathcal J$ which is attained by some critical point $u\in \mathcal C_0$. Let us pick a path $\gamma\in \mathcal P$ such that 
\[\bar{\mathcal J}(\gamma):=\sup_{s\in [0, 1]}\mathcal J(\gamma(s))<d+\eps.\]
Then $u(s, t, x)$ be the solution to the parabolic equation \eqref{eq: parabolic-equation} with initial condition $u(s, 0, \cdot) = \gamma(s)$. 
\begin{claim}\label{claim: smoothness}
    $u(s, t, x)\in \mathcal P$ exists for all time $t$, and $\bar{\mathcal J}(u(t))$ is decreasing in $t$. 
\end{claim}
\begin{proof}[Proof of Claim~\ref{claim: smoothness}]
	Short-time existence and the fact that $u(s, t, x)\in \mathcal P$ follows from Proposition~\ref{prop: short-time-existence} and the same arguments as in \cite[Theorem C]{Tso}. We will show that the solutions exist for all time. By Theorem~\ref{thm: estimate-parabolic}, it suffices to bound $|u(0, t)|$ from above and below uniformly in $t$. Let $v:B_r(0)\to \mathbb R$ solve $(v^{\star})^k\det D^2v = \delta$ for some $\delta<\eta$, and $v|_{\partial B_r(0)} = 0$, then $v$ is a supersolution of \eqref{eq: parabolic-equation}, and by the comparison principle, we have $|u(0, t)|>|v(0)|$. 
	Now we bound $u$ from below. The evolution equation for $u_t$ is
	\[u_{tt} = u^{ij}(u_t)_{ij}+k\frac{\langle x, \nabla u_t\rangle-u_t}{u^{\star}}-\frac{F_u(x, u)}{F(x, u)}u_t\]
	and since $\frac{F_u}{F}$ and $\frac{1}{u^{\star}}$ are both bounded, it follows that $|u_t|\leq Ce^{At}$ by the maximum principle. Hence $u$ is bounded from below on all finite time interval. The fact that $\overline{\mathcal J}(u(t))$ is decreasing in $t$ is clear. 
\end{proof}
Now let us define the closed sets $I_t\subset [0, 1]$ by 
\[I_t:=\{s\in [0, 1]: \mathcal J(u(s, t))\geq c\}\subset [0, 1],\]
then by the monotonicity of $\mathcal J(u(t, \cdot))$ along the parabolic equation~\eqref{eq: parabolic-equation}, the sets $I_t$ are a nested family of closed sets in $[0, 1]$, which means the intersection
\[I = \cap_{t>0}I_t\]
is non-empty. It follows that there exist an long-time solution $u(t, x)$ to the equation~\eqref{eq: parabolic-equation} such that 
\[\mathcal J(u(t, \cdot))\geq d \text{ for all }t>0. \]
By the monotonicity of $\mathcal J$, we have that for any $\eps>0$, we can pick $T$ sufficiently large, such that
\begin{equation}\label{eq: change-of-energy-formula}
\int_{T}^{\infty}\left(-\frac{d}{d t}\mathcal J(u(t, \cdot))\right)\,dt = \int_T^{\infty}\int_{\Omega}\left((u^{\star})^k\det D^2u-F(x, u)\right)\left(\log ((u^{\star})^k\det D^2u)-\log F(x, u)\right)\, dt\leq \eps. 
\end{equation}

\begin{claim}\label{claim: boundedness}
    For sufficiently large $t$, we have an estimate 
    \[\|u(t)\|_{L^{\infty}}\leq C\] uniformly in $t$. 
\end{claim}
\begin{proof}[Proof of Claim~\ref{claim: boundedness}]
We will follow the same line of argument as in \cite[Theorem C]{Tso}. Let us denote \[G:= (u^{\star})^k\det D^2u, \]
then by \eqref{eq: change-of-energy-formula} and the mean value theorem, we know that for every interval $[m, m+1]$ sufficiently large, we can find $t_m\in [m, m+1]$ such that
\[\int_\Omega (G(t_m)-F(t_m))(\log G(t_m)-\log F(t_m))\leq \eps. \] 
From now on we will fix such a $t_m$, and to simplify notation, we will assume all quantities are evaluated at chosen time $t_m$ and supress the dependence on $t_m$.
 
Let $\alpha>0$ be given by $e^{-\alpha} = 1-\frac{\theta}{2}$, and define $S\subset \Omega$ to be the set
\[S = \{|\log G-\log F|\leq \alpha\}\]
then 
\[\eps \geq \int_{\Omega\setminus S}(G-F)(\log G-\log F)\geq \alpha\int_{\Omega\setminus S}|G-F| = \alpha\int_{\Omega\setminus S}(e^{\log G-\log F}-1)F\geq \alpha(e^\alpha-1)\int_{\Omega\setminus S}F.\]
and since $F\geq \eta>0$, we have \[|\Omega\setminus S|\leq \frac{\eps}{\alpha(e^{\alpha}-1)\eta}.\] 

\[J(u) = \int_S\left(\frac{(-u)G}{n+k+1}-\int_u^0F(x, s)\,ds\right)+\int_{\Omega\setminus S}\left(\frac{(-u)G}{n+k+1}-\int_u^0F(x, s)\,ds\right)\]
For the first term, one has 
\begin{align}
\int_{S}\left(\frac{(-u)G}{n+k+1}-\int_u^0F(x, s)\,ds\right)&\geq \frac{1}{n+k+1}\int_S(-u)(G-F)+\frac{\theta}{n+k+1}\int_{S}(-u)F-C\\
&\geq \frac{\theta}{2(n+k+1)}\int_S(-u)F-C \\
&\geq c\theta\|u\|_{L^{\infty}}^{n+k+1}-C
\end{align}
For the second term, one has 
\begin{align}
\int_{\Omega\setminus S}\left(\frac{(-u)G}{n+k+1}-\int_u^0F(x, s)\,ds\right)&\geq \frac{1}{n+k+1}\int_{\Omega\setminus S}(-u)G-\frac{1-\theta}{n+k+1}\int_{\Omega\setminus S}(-u)F-C\\
&\geq -\frac{\eps}{\alpha(e^{\alpha}-1)}\|u\|_{L^{\infty}}-C
\end{align}
It follows that 
\[C\geq \mathcal J(u)\geq c\theta\|u\|_{L^{\infty}}^{n+k+1} -\frac{\eps}{\alpha(e^{\alpha}-1)}\|u\|_{L^{\infty}}-C.\]
This shows that $\|u(t_m)\|_{L^{\infty}}\leq C$ independent of $m$. To upgrade this bound to holds for all $t$ sufficiently large, we can follow the same argument as in \cite[Theorem C]{Tso}. This proves the claim. 
\end{proof}
With this claim, we know that $|u(t)|$ is bounded for all time, therefore by the estimates of Theorem~\ref{thm: estimate-parabolic}, it follows that $u(t, \cdot)$ is bounded uniformly in $C^{k, \alpha}$ independent of $t$. Moreover, there exist $t_i \to \infty$ and $\eps_i\to 0$, such that 
\[\left|\left.\frac{d}{d t}\right|_{t = t_i}\mathcal J(u(t))\right|\leq \eps_i.\]
 Therefore taking a subsequence as $i\to \infty$, we obtain a limit $u_{\infty}\in C^{\infty}(\overline\Omega)$ which solves \eqref{eq: elliptic-equations-F} and $d\leq\mathcal J(u_{\infty})(u)\leq d+\eps$. 
\end{proof}

\begin{corr}\label{cor: supercritical}
Let $p> n+k>0$, then the equations \eqref{eq: degen-dirichlet-problem} admit a non-zero convex solution $u\in C^{2, \alpha}(\overline\Omega)$. 
\end{corr}
\begin{proof}
Let 
\[\mathcal J_{\eps}(u)=H_{n+k}(u)-\frac{1}{p+1}\int_\Omega((\eps-u)^{p+1}-\eps^{p+1}),\]
then by the \ref{ineq: sobolev-ineq}, we have
\begin{equation}\label{eq: supercrit-lower-bdd}
\mathcal J_{\eps}(u)\geq a\|u\|_{L^{\infty}}^{n+k+1}-A\|u\|_{L^{\infty}}^{p+1}-O(\eps)\|u\|_{L^{\infty}}\end{equation}
and therefore it's easy to check that $F_{\eps}(x, u) = (\eps-u)^{p}$ satisfies the assumptions of Theorem~\ref{thm: supercritical} with $\mathcal J(u)> \frac{a}{2}(\frac{p-n-k}{p+1})\left(\frac{(n+k+1)a}{(p+1)A}\right)^{\frac{n+k+1}{p-n-k}}$ for $\|u\|_{L^{\infty}} = \left(\frac{(n+k+1)a}{(p+1)A}\right)^{\frac{1}{p-n-k}}$. Theorem~\ref{thm: supercritical} allows us to solve the equations
\[(u_{\eps}^{\star})^k\det D^2u_{\eps} = (\eps-u_{\eps})^p\]
for $p>n+k$, and moreover by equation~\eqref{eq: supercrit-lower-bdd}, we have 
$\mathcal J_{\eps}(u_{\eps}) \geq c>0$, and since 
\[\mathcal J_{\eps}(u_{\eps}) = \int_\Omega\frac{(-u_{\eps})(\eps-u_{\eps})^p}{n+k+1}-\frac{(\eps-u_{\eps})^{p+1}-\eps^{p+1}}{p+1}, \]
this gives us uniform bounds on $\|u_{\eps}\|_{L^{\infty}}$ both above and below. By the estimates in Theorem~\ref{thm: elliptic estimates}, $u_{\eps}$ must converge along some subsequence to some non-zero $u\in C^{\infty}(\Omega)\cap C^{0, 1}(\overline\Omega)$ which is a solution of \eqref{eq: degen-dirichlet-problem}. By the arguments of \cite[Theorem 1.3]{Savin} and \cite[Theorem 1.2]{Le-Savin} we get $u\in C^{2, \alpha}(\overline\Omega)$. 
\end{proof}

\subsection{An eigenvalue problem}
Now we treat the scale-invariant equation as an eigenvalue problem. As before, we will assume that $n+k\geq 0$, and consider the equation, 
\begin{equation}\label{eq: eigenvalue-problem}
\begin{cases}
(u^{\star})^k\det D^2u = \lambda |u|^{n+k} &\text{ in }\Omega\\
u=0 &\text{ on }\partial\Omega.
\end{cases}
\end{equation} 

\begin{theorem}[Existence and uniqueness of eigenfunctions]\label{thm: eigenvalue-problem}
	There exist $u\in C^{2, \alpha}(\overline\Omega)\cap C^{\infty}(\Omega)$ which is a solution to \eqref{eq: eigenvalue-problem} for $\lambda = \underline\lambda_{n+k+1}$. Moreover, if $(\lambda', u')$ is another pair of such solution, then $\lambda' = \lambda$ and there exist $c>0$ such that $u = cu'$. 
\end{theorem}

\begin{proof}
	We first consider the case when $n+k>0$. Consider the  family of equations depending on a parameter $s>0$. 
	\begin{equation}\label{eq: s-regularized}
	\begin{cases}
	(u^{\star})^k\det D^2u = (1-s u)^{n+k} &\text{ in }\Omega\\
	u=0 &\text{ on }\partial\Omega.
	\end{cases}
	\end{equation} 
	and define the set $I\subset \mathbb R_{\geq 0}$ by
	\[I :=\{s\geq 0: \eqref{eq: s-regularized}\text{ admits a solution}\} \]
	For $s\in I$, we'll denote $u_{s}$ a solution of \eqref{eq: s-regularized}. It clear that if $s'\in I$, then $u_{s'}$ is a subsolution of \eqref{eq: s-regularized} for all $s<s'$, therefore $I$ must be an interval starting from $0$. Also if $s'\in I$, then $2u_{s}$ is a subsolution of \eqref{eq: s-regularized} for all $s<s'+\eps$ for some $\eps>0$, therefore $I$ is open. It follows that either $I = \mathbb R_{\geq 0}$, or $I = [0, \tilde s)$ for some $\tilde s>0$. 
	
	\begin{claim}\label{claim: def-of-I}
		$I = [0, \tilde s)$, and $\tilde s \leq \underline\lambda_{n+k+1}^{\frac{1}{n+k}}$. 
	\end{claim}
	\begin{proof}
		Given any $s<\tilde{s}$, let $u$ be the solution to \eqref{eq: s-regularized}. Then for $p<n+k$, consider the equations 
		\begin{equation}\label{eq: s-regularized-p}
		\begin{cases}
		(u^{\star})^k\det D^2u = (1-s u)^{p} &\text{ in }\Omega\\
		u=0 &\text{ on }\partial\Omega.
		\end{cases}
		\end{equation} 
		By Theorem~\ref{thm: subcritical}, \eqref{eq: s-regularized-p} admits a unique solution $u_p$ which moreover minimizes the functional 
		\[\mathcal J_p(u):= H_{n+k}(u)-\frac{s^{-1}}{p+1}\int_\Omega(1-s u)^{p+1}. \]
		Also since $u$ is a subsolution to \eqref{eq: s-regularized-p}, by \cite[Theorem 7.1]{CNS} and Theorem~\ref{thm: uniqueness}, we know that $|u_p|\leq |u|$ is uniformly bounded, therefore by the estimates \ref{thm: estimate-parabolic}, it follows that 
		\[\inf \mathcal J_p = \mathcal J_p(u_p)\geq -C\] 
		is uniformly bounded away from $-\infty$. This implies that the functional
		\[\mathcal J_{n+k}(u):= H_{n+k}(u)-\frac{s^{-1}}{n+k+1}\int_\Omega(1-s u)^{n+k+1} \]
		is bounded below as well. However, we know that
		\[ H_{n+k}(u)-\frac{s^{n+k}}{n+k+1}\int_\Omega|u|^{n+k+1}\geq \inf J_{n+k}>-\infty\]
		which implies $s^{n+k}\leq \underline{\lambda}_{n+k+1}$, which proves the claim. 
	\end{proof}
	
	As $s\to \tilde{s}$, we claim that $\sup_{\Omega}|u_{s}|$ must tend to $\infty$. For otherwise by \ref{eq: parabolic-equation-with-boundary}, we can extract a subsequence that converges to a function $u_{\tilde s}$, which is a solution of \eqref{eq: s-regularized}, which contradicts the definition of $\tilde s$. Now we consider the sequence of rescaled solutions
	\[v_{s} := \frac{u_{s}}{\|u_{s}\|}.\]
	Then $v_{s}$ satisfy the equations
	\begin{equation}
	\begin{cases}
	\det D^2v_{s} = (\|u_{s}\|_{L^{\infty}}^{-1}-s v_{s})^{n+k}(u_{s}^{\star})^{-k} &\text{ in }\Omega\\
	v_{s} = 0&\text{ on }\partial \Omega. \end{cases}
	\end{equation}
	and $\|u_{s}\|_{C^0} = 1$. By the first part of Theorem~\ref{thm: elliptic estimates}, we get a $C^1$ bound $\|u_{s}\|_{C^1}\leq C$, and we can extract a subsequence of $u_{s}$ that converges to uniformly a limit $u$, by the interior estimates in Theorem~\ref{thm: elliptic estimates}, $u$ solves the equation~\eqref{eq: eigenvalue-problem} in $\Omega$ with $\lambda =\tilde s^{n+k}$. By the arguments of \cite[Theorem 1.3]{Savin} and \cite[Theorem 1.2]{Le-Savin} we get $u\in C^{2, \alpha}(\overline\Omega)$ for some $\alpha>0$. This proves the existence claim when $k>-n$. By \eqref{eq: def-of-lambda-n+k+1}, we must have that $\lambda =\tilde s^{n+k}\geq \underline\lambda_{n+k+1}$ which together with Claim~\ref{claim: def-of-I}, gives that $\lambda = \underline\lambda_{n+k+1}$. 
	
	In the case when $k = -n$, let $u_{\eps}$ be the solution of the eigenvalue problem \eqref{eq: eigenvalue-problem} for $k = -n+\eps$, normalized so that $\|u_{\eps}\|_{C^0} = 1$. It's clear that $\lambda_{\eps}\leq C$ uniformly in $\eps$, hence $\det D^2u_{\eps}\leq C(1+|\nabla u_{\eps}|)^n$, and by Theorem~\ref{thm: elliptic estimates}, we have a uniform gradient estimate $\|u_{\eps}\|_{C^1}\leq C$. Hence we can extract a subsequence that converge uniformly to a convex function $u$, and the same arguments as above shows $u\in C^{2, \alpha}(\overline\Omega)\cap C^{\infty}(\Omega)$ for some $\alpha>0$. In fact, in this case the Schauder estimates implies that $u\in C^{\infty}(\overline\Omega)$.

	The only thing left to prove is uniqueness. Suppose that there exist two pairs $(\lambda_1, u_1)$ and $(\lambda_2, u_2)$ which solves \eqref{eq: eigenvalue-problem}. Without loss of generality, we can assume that $\lambda_1\leq\lambda_2$, hence 
	\begin{align}
	0&\leq (\det D^2u_2)^{\frac 1 n}-(\det D^2u_1)^{\frac 1 n}
	-\lambda_1^{1/n}\left(\frac{(-u_2)^{\frac{n+k}{n}}}{(u_2^{\star})^{\frac{k}{n}}}-\frac{(-u_1)^{\frac{n+k}{n}}}{(u_1^{\star})^{\frac{k}{n}}}\right) \\
	& \leq \frac{1}{n}(\det D^2u_1)^{\frac{1}{n}}u_1^{ij}(u_2-u_1)_{ij}
	-\lambda_1^{1/n}\left(\frac{(-u_2)^{\frac{n+k}{n}}}{(u_2^{\star})^{\frac{k}{n}}}-\frac{(-u_1)^{\frac{n+k}{n}}}{(u_1^{\star})^{\frac{k}{n}}}\right). 
	\end{align}
	where the second line follows from the concavity of $M \mapsto (\det M)^{\frac 1 n}$. 
	Notice that any scaling $(\lambda_2, tu_2)$ for $t>0$ also solves \eqref{eq: eigenvalue-problem}, therefore we can replace $u_2$ by one of it's scalings $t_0u_2$ where 
	\[t_0 =\sup\{t>0: tu_2<u_1 \text{ in }\Omega\}. \]
	After replacing $u_2$ by $t_0u_2$, we have $u_2\leq u_1$ and either there exist $x_0\in \Omega$ for which $u_2(x_0) = u_1(x_0)$, or there exist $x_0\in \partial \Omega$ for which $\langle x, \nabla(u_2-u_1)\rangle(x_0) = 0$. If the first case is true, then by the strong maximum principle $u_1\equiv u_2$ and we are done. In the second case, we can apply the Hopf lemma as in Theorem~\ref{thm: uniqueness}. We've proven the Theorem. 
\end{proof}
We have essentially proven Theorem~\ref{thm: main-thm}. 
\begin{proof}[proof of Theorem~\ref{thm: main-thm}]
	This follows from Corollary~\ref{cor: subcritical}, Corollary~\ref{cor: supercritical} and Theorem~\ref{thm: eigenvalue-problem} and by scaling. 
\end{proof}

\section{Relationship to an optimal transport problem}\label{section: tranformation}
In this section, we describe a transformation that transforms the solution of the Dirichlet problems \eqref{eq: main-equations} to that of a solution of a second boundary value problem, which has an interpretation of an optimal transport problem. This transformation arises from the duality of affine spheres \cite{Calabi}. (see also \cite[Section 5]{Klartag} for a nice exposition.) We now describe the transformation below. 

Suppose $u\in \mathcal C_0$, then let us denote RGraph$(u)$ to be {\bf radial graph} of the function $\frac{1}{-u}$,
\[\mathbb R^n\times \mathbb R_+\supset \text{RGraph}(u) := \left\{\frac{1}{-u(x)}(x, 1): x\in \Omega\right\}. \]
In this setting, RGraph$(u)$ is actually the graph of a convex function $\varphi: \mathbb R^n\to \mathbb R_+$, and moreover we have the relationship $\nabla\varphi(\mathbb R^n) = \Omega^{\circ}$. (see \cite[Section 5]{Klartag}) The following proposition describes the relationship between the equation satisfied by $u$ and the equation satisfied by $\varphi$. 
\begin{prop}\label{prop: tranform-equation}
	Let $u$ and $\varphi$ be described by above, then we have
\begin{equation}\label{eq: u-phi-relation}
	\frac{\det D^2u(x)}{(u^{\star})^{n+2}(x)} = \varphi^{n+2}\left(\frac{x}{-u(x)}\right)\det D^2\varphi\left(\frac{x}{-u(x)}\right). 
\end{equation}
\end{prop}
\begin{proof}
Let $y = \frac{x}{-u(x)}$ be coordinates on $\mathbb R^n$, then by the definition of $\varphi$, we have $\varphi(y) = \frac{1}{-u(x)}$ and $x = \frac{y}{\varphi(y)}$. Then by differentiating, we get
\[\frac{\partial x_i}{\partial y_j} = \frac{\delta_{ij}}{\varphi}-\frac{y_i}{\varphi^2}\frac{\partial \varphi}{\partial y_j} = \frac{\delta_{ij}}{\varphi}-\frac{x_i}{\varphi}\frac{\partial \varphi}{\partial y_j}  =-u(\delta_{ij}-x_i\frac{\partial \varphi}{\partial y_j}),\]
and by chain rule, we get
\[\frac{\partial \varphi}{\partial y_j} = \frac{\partial \varphi}{\partial x_i}\frac{\partial x_i}{\partial y_j} = \frac{u_i}{u}(x_i\frac{\partial \varphi}{\partial y_j}-\delta_{ij}),\]
which implies 
\[\frac{\partial \varphi}{\partial y_j} = -\frac{u_j}{u-\langle x, \nabla u\rangle} = \frac{u_j}{u^{\star}}.\]
taking another derivative and applying chain rule, we get
\begin{align}
\frac{\partial^2 \varphi}{\partial y_i\partial y_j} = \frac{-u}{u^{\star}}\left(u_{jk}-\frac{u_jx_mu_{mk}}{u^{\star}}\right)\left(\delta_{ki}-x_k\frac{u_i}{u^{\star}}\right)
\end{align}
taking a determinant gives what we desired. 
\end{proof}
\begin{remark}
It is known that the graph of a function $u$ is an elliptic affine sphere with center at the origin if and only if $u$ satisfies
\[\frac{\det D^2u}{(u^{\star})^{n+2}} = c\]
Therefore from Proposition~\ref{prop: tranform-equation} and this characterization of affine spheres, we see that Graph$(u)$ is a elliptic affine sphere if and only if the graph of the Legendre tranform of $\varphi$ is also an elliptic affine sphere. This is essentially the duality of affine sphere described in \cite{Calabi}. 
\end{remark}
More generally, we can see that if $u$ is a solution of \eqref{eq: main-equations}, then by \eqref{eq: u-phi-relation}, $\varphi:\mathbb R^n\to \mathbb R$ solves the second boundary value problem. 
\begin{equation}\label{eq: second boundary problem}
\begin{cases}
\det D^2\varphi = \lambda\frac{(-\varphi^{\star})^{n+2+k}}{\varphi^{n+2+p}}\\
\nabla\varphi(\mathbb R^n) = \Omega^{\circ}. 
\end{cases}
\end{equation} 

\begin{remark}
In the case $k =-n-2$, the numerator on the right-hand-side of equation~\eqref{eq: second boundary problem} does not appear. In this case, equation~\eqref{eq: second boundary problem} was introduced and solved by Klartag in \cite{Klartag}, who proved the existence and uniqueness of solutions up to translations. It is in general not known whether Klartag's solutions can be transformed to a solution of a Dirichlet problem \eqref{eq: main-equations}. 
\end{remark}

Based on the above discussion, an immediate consequence of our existence results for the Dirichlet problem \eqref{eq: main-equations} is the following
\begin{prop}
	Let $n+k>0$, then the second boundary value problem \eqref{eq: second boundary problem} admits a nontrivial solution for all $p\neq n+k$. For $p=n+k$, \eqref{eq: second boundary problem}  admits a solution for $\lambda = \underline\lambda_{n+k+1}$. 
\end{prop}
\begin{proof}
	This follows from Theorem~\ref{thm: main-thm} and the discussion above. 
\end{proof}

\begin{remark}
More generally, \eqref{eq: second boundary problem} is equivalent to the condition \[\nabla\varphi_{\star}\left(\frac{\lambda dx}{\varphi^{n+2+p}(x)}\right)  = \frac{dy}{(-\varphi^{\star}(y))^{n+2+k}}. \] From this we see that $\nabla\varphi$ solves an optimal tranport problem from $\mathbb R^n$ to $\Omega^{\circ}$, where the domain is
\[d\mu = \frac{\lambda dx}{\varphi^{n+2+p}(x)}\] 
and target measure is
\[d\nu = \frac{dy}{(-\varphi^{\star}(y))^{n+2+k}}. \]
Hence $\varphi$ has an interesting interpretation as the solution of an optimal transport problem whose domain and target measure are coupled to the potential $\varphi$ and its Legendre transform $\varphi^{\star}$. 
\end{remark}

\paragraph{Acknowledgements:} We would like to thank John Loftin for useful discussions and for pointing us to the paper \cite{Klartag}. F.T. would like to thank D.H. Phong for his interest and encouragement.

\Addresses
\end{document}